\documentclass[10pt]{article}
\usepackage[utf8]{inputenc}
\usepackage[english]{babel}
\usepackage{amsmath,amsthm,amssymb}
\usepackage{enumitem}
\usepackage{authblk}
\usepackage{xcolor}
\usepackage[symbol]{footmisc}
\usepackage{hyperref}
\hypersetup{colorlinks=true, linkcolor=blue, urlcolor=black}

\newtheorem{theorem}{Theorem}[section]

\newtheorem{corollary}[theorem]{Corollary}
\newtheorem{lemma}[theorem]{Lemma}
\newtheorem{proposition}[theorem]{Proposition}

\theoremstyle{definition}
\newtheorem{definition}[theorem]{Definition}

\theoremstyle{remark} \theoremstyle{remark}
\newtheorem{remark}[theorem]{Remark}
\newtheorem{example}[theorem]{Example}

\numberwithin{equation}{section}

\newcommand{\R}{\mathbb{R}}
\newcommand{\C}{\mathbb{C}}
\newcommand{\X}{\mathfrak{X}}
\newcommand{\g}{\mathfrak{g}}
\newcommand{\h}{\mathfrak{h}}
\newcommand{\Liek}{\mathfrak{k}}
\newcommand{\z}{\mathfrak{z}}
\newcommand{\m}{\mathfrak{m}}
\newcommand{\ad}{\operatorname{ad}}
\newcommand{\Ker}{\operatorname{Ker}}
\newcommand{\D}{\mathcal{D}}
\newcommand{\A}{\mathcal{A}}
\newcommand{\f}{\varphi}
\newcommand{\Lie}{\mathcal{L}}

\newcommand{\e}{\varepsilon}
\DeclareMathOperator{\Ric}{Ric}
\DeclareMathOperator{\Kill}{B}

\begin{document}
	\title{\textbf{On anti-quasi-Sasakian manifolds of maximal rank}}
	\author{Dario Di Pinto}
	\date{}	
	\maketitle
	
\begin{abstract}
	We discuss the existence of invariant anti-quasi-Sasakian (aqS) structures of maximal rank on compact homogeneous Riemannian manifolds and on nilpotent Lie groups. In the former case we obtain a non-existence result, while in the latter case we provide a complete classification. We also show that every compact aqS manifold has nonvanishing second Betti number.
\end{abstract}
\bigskip
	
{\noindent \small \textbf{MSC (2020):} 53C15, 53D15, 53C25, 53C30, 22E25. 
\smallskip
		
\noindent	\textbf{Keywords and phrases:} anti-quasi-Sasakian manifolds, compact homogeneous spaces, generalized flag manifolds, weighted Heisenberg Lie group, nilpotent Lie algebras.	
}

\section{Introduction}
In the recent paper \cite{DiP.D.}, the authors introduced a new class of almost contact metric manifolds, called \textit{anti-quasi-Sasakian} (aqS for short), whose characteristic feature is to be non-normal almost contact metric manifolds, locally fibering along the Reeb vector field $\xi$ onto K\"ahler manifolds endowed with a closed $2$-form of type (2,0). Precisely, an almost contact metric manifold $(M,\f,\xi,\eta,g)$ is said to be anti-quasi-Sasakian if 
	\begin{equation}\label{eq:def_intro}
		d\Phi=0,\qquad N_\f=2d\eta\otimes\xi,
	\end{equation}
where $\Phi=g(\cdot,\f\cdot)$ is the fundamental $2$-form associated with the almost contact metric structure, and $N_\f=[\f,\f]+d\eta\otimes\xi$. An anti-quasi-Sasakian manifold is normal if and only if $d\eta=0$, in which case it is cok\"ahler. Actually, the second condition in \eqref{eq:def_intro} expresses the fact that the rank of $\eta$ measures how far the structure is from being normal. The same condition implies that the $2$-form $d\eta$ is $\f$-anti-invariant, namely $d\eta(\f X,\f Y)=-d\eta(X,Y)$ for every $X,Y\in \X(M)$. In fact, the name of such manifolds is due to the fact that, within the class of transversely K\"ahler almost contact metric manifolds, quasi-Sasakian and anti-quasi-Sasakian manifolds are characterized, respectively, by the $\f$-invariance and the $\f$-anti-invariance of $d\eta$. In the case of aqS manifolds, this forces the rank of $\eta$ to be of type $4p+1$, in the sense that $\eta\wedge(d\eta)^{2p}\neq0$ and $(d\eta)^{2p+1}=0$. Consequently, aqS manifolds of maximal rank, i.e. those for which $\eta$ is a contact form, have dimension $4n+1$. 

In \cite{DiP.D.} a Boothby-Wang construction is provided, showing that anti-quasi-Sasakian manifolds can be obtained as principal circle bundles over K\"ahler manifolds endowed with a closed integral $2$-form $\omega$ of type $(2,0)$. Whenever $\omega$ is nondegenerate, the aqS structure on the $\mathbb{S}^1$-bundle has maximal rank. This is the case when the base manifold is hyperk\"ahler and $\omega$ is one of the three K\"ahler forms. Recently, the existence of compact hyperk\"ahler manifolds with integral K\"ahler forms has been showed in \cite{Cortes}.
	
The fact that for an aqS manifold of maximal rank the transverse geometry with respect to $\xi$ is given by a K\"ahler structure with a nondegenerate closed $(2,0)$-form can have geometric obstructions, such as homogeneity and compactness.  By homogeneous anti-quasi-Sasakian manifolds we mean aqS manifolds whose $(\f,\xi,\eta,g)$-structure is invariant under the transitive action of a Lie group $G$. In \cite{DiP.D.} it is proved that there exist no connected homogeneous aqS manifold of maximal rank with $\Ric\ge0$. In Section \ref{Sec:comp.hom.}, we replace the assumption on the Ricci tensor field by requiring the compactness of the manifold, thus obtaining the following non-existence result:

\begin{theorem}\label{Thm:0intro}
	There exist no compact homogeneous anti-quasi-Sasakian manifolds of maximal rank.	
\end{theorem}
	
\noindent This is consequence of the fact that such manifolds should fiber onto compact, simply connected, homogeneous K\"ahler manifolds. These are the so-called generalized flag manifolds, which can be realized as quotients of a compact semisimple Lie group $G$ by the centralizer of a torus in $G$. Then, the result in Theorem \ref{Thm:0intro} is obtained by showing that generalized flag manifolds do not admit any invariant closed $2$-form of type $(2,0)$.
	
A topological obstruction to the existence of anti-quasi-Sasakian structures on a compact manifold $M$, even in the inhomogeneous case and for any rank, is represented by the nonvanishing of the second Betti number, as showed in Proposition \ref{Prop:Betti}.  As a consequence, every odd dimensional sphere does not admit any aqS structure. 

In Section \ref{Sec:nilpotent} we discuss the case of anti-quasi-Sasakian structures on nilpotent Lie groups. Invariant aqS structures can be defined in a natural way on $2$-step nilpotent Lie groups $H^{4n+1}_\lambda$, called weigthed Heisenberg Lie groups, for which the commutators in the Lie algebra  $\h^{4n+1}_\lambda$ depend on some weights $\lambda=(\lambda_1,\dots,\lambda_n)$. In fact we show the following:

\begin{theorem}\label{Thm:1intro}
	A  nilpotent Lie algebra $\g$ admits an anti-quasi-Sasakian structure of maximal rank if and only if $\g$ is isomorphic to a weighted Heisenberg Lie algebra $\h^{4n+1}_\lambda$, for some non-zero weights $\lambda_1,\dots,\lambda_n$. 
\end{theorem}

\noindent Therefore, a compact nilmanifold $G/\Gamma$ of dimension $4n+1$ admits an anti-quasi-Sasakian structure of maximal rank induced by a left-invariant aqS structure on $G$ if and only $G$ is isomorphic to $H^{4n+1}_\lambda$.
The proof of Theorem \ref{Thm:1intro} lies on the fact that nilpotent Lie algebras endowed with a transversely K\"ahler almost contact metric structure of maximal rank are $1$-dimensional central extensions of an abelian Lie algebra. Actually, this argument also applies to nilpotent Lie algebras endowed with a quasi-Sasakian structure of maximal rank, allowing to show that they are all isomorphic to a weighted Heisenberg Lie algebra $\h^{2n+1}_\lambda$. This provides a generalization of the well known fact that the only nilpotent Lie algebra admitting a Sasakian structure is the standard Heisenberg Lie algebra $\h^{2n+1}$ (see \cite{AFV,CM.dN.M.Y_Sas.Nil}). 
\\
		
\textbf{Acknowledgements:} The author would like to thank Oliver Goertsches for sharing some ideas underlying this paper, and also Giulia Dileo and Antonio Lotta for helpful comments. 
	
The author is members of INdAM - GNSAGA (Gruppo Nazionale per le Strutture Algebriche, Geometriche e le loro Applicazioni).

\section{Review of anti-quasi-Sasakian manifolds}
An \textit{almost contact manifold} is an odd dimensional smooth manifold $M^{2n+1}$ endowed with a $(1,1)$-tensor field $\f$, a vector field $\xi$, and a $1$-form $\eta$ satisfying
	$$\f^2=-I+\eta\otimes\xi,\qquad \eta(\xi)=1.$$
It follows that $\f\xi=0$ and $\eta\circ\f=0$, so that $\D:=\Ker\eta=\operatorname{Im}\f$ is a hyperplane distribution of constant rank $2n$. Moreover, the tangent bundle of $M$ splits as $TM=\langle\xi\rangle\oplus\D$, and the restriction $\f|_\D$ is an almost complex structure on $\D$, i.e. $\f|_\D^2=-I$. An almost contact manifold is \textit{normal} if the tensor field $N_\f:=[\f,\f]+d\eta\otimes\xi$ identically vanishes, where $[\f,\f]$ denotes the Nijenhuis torsion of $\f$, explicitely given by
	$$[\f,\f](X,Y)=[\f X,\f Y]+\f^2[X,Y]-\f[X,\f Y]-\f[\f X,Y]$$
for every $X,Y\in\X(M)$.
A Riemannian metric $g$ is said to be compatible with the $(\f,\xi,\eta)$-structure if $$g(\f X,\f Y)=g(X,Y)-\eta(X)\eta(Y),$$ in which case $(M,\f,\xi,\eta,g)$ is called an \textit{almost contact metric manifold} and the Reeb vector field $\xi$ turns out to be unitary and orthogonal to the horizontal distribution $\D$.
We will denote by $\Phi(X,Y):=g(X,\f Y)$ the fundamental $2$-form of the almost contact metric structure. Among all possible almost contact metric manifolds, several classes can be distinguished, the most studied ones being the following:
	
\begin{itemize}[itemsep=0.5pt]
	\item \textit{contact metric} if $d\eta=2\Phi$;
	\item \textit{Sasakian} if $d\eta=2\Phi$, $N_\f=0$;
	\item \textit{cok\"ahler} if $d\eta=0$, $d\Phi=0$, $N_\f=0$;
	\item \textit{quasi-Sasakian} if $d\Phi=0$, $N_\f=0$.
\end{itemize}

Analogous structures can be defined on a $(2n+1)$-dimensional Lie algebra $\g$, determining a left-invariant structure on a Lie group $G$ with Lie algebra $\g$. 
We refer the reader to \cite{Blair,BG,CM.dN.Y_cos,Blair_qS} for an extensive treatment of the above classes. Here we just point out that the contact metric condition implies that $\eta$ is a contact form, in the sense that $\eta\wedge(d\eta)^n$  never vanishes, where $2n+1$ is the dimension of the manifold $M$. More generally, one says that an almost contact structure $(\f,\xi,\eta)$ has constant rank $2p$ if $(d\eta)^p\neq0$ and $\eta\wedge(d\eta)^p=0$, and it has constant rank $2p+1$ if $\eta\wedge(d\eta)^{2p}\neq0$ and $(d\eta)^{2p+1}=0$.  In \cite{Blair_qS} Blair showed that quasi-Sasakian structures cannot have even rank, and in particular, cok\"ahler manifolds have minimal rank $1$, while Sasakian manifolds have maximal rank $2n+1$. In \cite{DiP.D.} the authors introduced the class of anti-quasi-Sasakian manifolds. Next we recall some basic facts, referring to the same paper for details. 
	
\begin{definition}
	An almost contact metric manifold $(M,\f,\xi,\eta,g)$ is called \textit{anti-quasi-Sasakian} (aqS for short) if 
	\begin{equation}
			d\Phi=0,\qquad N_\f=2d\eta\otimes\xi.
	\end{equation}
\end{definition}
	
\noindent  In particular, the second equation is referred as the \textit{anti-normal condition}. It implies that $d\eta(\xi,\cdot)=0$ and $d\eta$ is $\f$-anti-invariant, namely $d\eta(\f X,\f Y)=-d\eta(X,Y)$ for every $X,Y\in\X(M)$. This forces the rank of the structure to be of type $4p+1$. In particular, an almost contact structure is both normal and anti-normal if and only if it has minimal rank $1$, i.e. $d\eta=0$, and in fact the intersection between quasi-Sasakian and anti-quasi-Sasakian manifolds is precisely the class of cok\"ahler manifolds. The Reeb vector field $\xi$ of an aqS manifold is Killing. Furthermore, the Lie derivatives $\Lie_\xi\eta$, $\Lie_\xi d\eta$ and $\Lie_\xi\f$ all vanish, so that the structure tensor fields $\f,g,d\eta$ are locally projectable along the $1$-dimensional foliation generated by $\xi$; the induced structure on the space of leaves $M/\xi$ is a K\"ahler structure  endowed with a closed $2$-form of type $(2,0)$. In fact the following holds:

\begin{theorem}\label{Thm:Boothby-Wang1}
	Let $(M,\f,\xi,\eta,g)$ be a anti-quasi-Sasakian manifold such that $\xi$ is regular, with compact orbits. Then, $M$ is a principal circle bundle over a K\"ahler manifold $M/\xi$ endowed with a closed $2$-form $\omega$ of type $(2,0)$ .
	In particular, $\eta$ is a connection form on $M$ and its curvature form is given by $d\eta=\pi^*\omega$, where $\pi:M\to M/\xi$ is the bundle projection.
\end{theorem}
	Conversely:
\begin{theorem}\label{Thm:Boothby-Wang2}
	Let $(B,J,k)$ be a K\"ahler manifold endowed with an integral, closed $2$-form $\omega$ of type $(2,0)$. Then there exist a principal circle bundle $M$ over $B$ and a connection form $\eta$ on $M$ such that the curvature form is given by $d\eta=\pi^*\omega$.
	Moreover, $M$ is endowed with an anti-quasi-Sasakian structure $(\f,\xi,\eta,g)$.
\end{theorem}

Considering an anti-quasi-Sasakian manifold $(M,\f,\xi,\eta,g)$, the operators defined by $$A:=-\f\circ\nabla\xi,\qquad\psi:=A\f=-\nabla\xi$$ 
are skew-symmetric and satisfy the following identities:
\begin{equation}\label{eq:A.f.psi}
	A\f=\psi=-\f A,\quad \f\psi=A=-\psi\f,\quad \psi A=-\f A^2=-A\psi.
\end{equation} 
In particular, $A\xi=\psi\xi=0$, $\eta\circ A=\eta\circ\psi=0$, and $A^2=\psi^2$. They  determine the $2$-forms 
$$ \A=g(\cdot,A\cdot),\qquad \Psi=g(\cdot,\psi\cdot).$$
Taking into account also the fundamental $2$-form $\Phi$, one has a triplet of $2$-forms $(\A,\Phi,\Psi)$, satisfying
\begin{equation}\label{eq:closedness}
	d\A=0,\qquad d\Phi=0,\qquad d\eta=2\Psi.
\end{equation}
If $\eta$ has maximal rank, by the third identity in \eqref{eq:closedness} it follows that $\Psi$ is nondegenerate on $\D$, i.e. $\psi:\D\to\D$ is an isomorphism, and hence so is also $A|_\D=\f\psi|_\D$. Finally, being $\psi$ skew-symmetric, $\psi^2=A^2$ is a symmetric operator with nonpositive eigenfunctions. The case where both $A$ and $\psi$ define almost contact structures is characterized by the following:
	
\begin{theorem}[{\cite[Theorem 4.7]{DiP.D.}}]
	Let $(M,\f,\xi,\eta,g)$ be an anti-quasi-Sasakian manifold. Then, the following conditions are equivalent:
	\begin{enumerate}[label=(\roman*), noitemsep]
		\item $\psi^2=A^2=-I+\eta\otimes\xi$;
		\item $Sp(\psi^2)=\{0,-1\}$, with $0$ simple eigenvalue;
		\item $M$ has constant $\xi$-sectional curvature $1$, i.e. $K(\xi,X)=1$ for every $X\in\Gamma(\D)$.
	\end{enumerate}
\end{theorem}

\noindent If any of (i)-(iii) holds, it turns out that $(A,\xi,\eta,g)$ is an anti-quasi-Sasakian structure, while $(\psi,\xi,\eta,g)$ is Sasakian. We say that $(A,\f,\psi,\xi,\eta,g)$ is a double aqS-Sasakian structure according to the following definition. 

\begin{definition}\label{Def:double aqS-S}
	A \textit{double aqS-Sasakian manifold} is a $(4n+1)$-dimensional manifold $M$ endowed with three almost contact metric structures $(\f_i,\xi,\eta,g)$ ($i=1,2,3$) such that $\f_1\f_2=\f_3=-\f_2\f_1$,
	with associated fundamental $2$-forms satisfying
	\begin{equation}\label{eq:daqS-S}
		d\Phi_1=0,\quad d\Phi_2=0,\quad d\eta=2\Phi_3.
	\end{equation}
\end{definition}

\noindent In fact, equation \eqref{eq:daqS-S} implies that $(\f_1,\xi,\eta,g)$ and $(\f_2,\xi,\eta,g)$ are anti-quasi-Sasakian, and $(\f_3,\xi,\eta,g)$ is Sasakian, thus justifying the given name for such manifolds. Furthermore, $M$ locally fibers onto a hyperk\"ahler manifold. For $n=1$, double aqS-Sasakian $5$-manifolds coincide with contact Calabi-Yau manifolds introduced in \cite{TV}, whose structure is a special type of $K$-contact hypo $SU(2)$-structures \cite{dAFFU}.
\\

We conclude this section with a technical lemma which will be used later on.

\begin{lemma}\label{Lemma:omega_basis}
	Le $(M^{4n+1},\f,\xi,\eta,g)$ be an anti-quasi-Sasakian manifold of maximal rank. Then, there exists a local orthogonal frame of $TM$ of type $\{\xi,e_i,Ae_i,\f e_i,\psi e_i\}$, where $i=1,\dots,n$, $\psi^2e_i=-\lambda_i^2e_i$ and $\lambda_i\neq0$.
\end{lemma}
\begin{proof}
	Since the structure has maximal rank, $\psi:\D\to\D$ is an isomorphism, so that the symmetric operator $\psi|^2_\D$ has non zero eigenvalues.  For each eigendistribution $\D_{-\mu^2}$, let us fix a non zero eigenvector field $X\in\Gamma(\D_{-\mu^2})$. Applying \eqref{eq:A.f.psi} and the skew-symmetry of $A,\f,\psi$, one can easily verify that $X,AX,\f X,\psi X$ are mutually orthogonal eigenvector fields associated to the same eigenfunction $-\mu^2$. If $\dim\D_{-\mu^2}>4$, there exists $Y\in\Gamma(\D_{-\mu^2})$ which is orthogonal to $X,AX,\f X,\psi X$. Then, $Y,AY,\f Y,\psi Y$ are orthogonal to each other and also to $X,AX,\f X,\psi X$. Iterating the argument one obtains a local othogonal basis of each eigendistribution, and hence of $\D=\langle\xi\rangle^\perp$.
\end{proof}

\begin{remark}\label{Rmk:basis}
	Up to scaling, we may assume $e_i$ unitary for every $i=1,\dots,n$. Then, one has that $\|\f e_i\|^2=1$ and $\|Ae_i \|^2=\|\psi e_i\|^2=\lambda_i^2$. Setting $e_{n+i}=\frac1{\lambda_i}Ae_i$, $e_{2n+i}=\f e_i$ and $e_{3n+i}=\frac1{\lambda_i}\psi e_i$, we obtain a local orthonormal frame $\{\xi,e_i,e_{n+i},e_{2n+i},e_{3n+i}\}$. Moreover, denoting by $\e_l$ the dual form of $e_l$ ($l=1,\dots,4n$), the  $2$-forms $\A,\Phi,\Psi$ can be expressed as follows:
	\begin{align}
		\A&=-\sum_{i=1}^{n}\lambda_i(\e_i\wedge\e_{n+i}+\e_{2n+i}\wedge\e_{3n+i}),
		\label{eq:A}\\
		\Phi&=-\sum_{i=1}^{n}(\e_i\wedge\e_{2n+i}+\e_{3n+i}\wedge\e_{n+i}),
		\label{eq:Phi}\\
		\Psi&=-\sum_{i=1}^{n}\lambda_i(\e_i\wedge\e_{3n+i}+\e_{n+i}\wedge\e_{2n+i})
		\label{eq:Psi}.
	\end{align} 
\end{remark}

\section{Non-existence results for compact aqS manifolds}
\label{Sec:comp.hom.}
		
Recall that a K\"ahler manifold $(M^{2n},J,k)$ is called \textit{homogeneous} if there exists a Lie group $G$ acting transitively on $M^{2n}$ by holomorphic isometries, i.e. preserving the K\"ahler structure $(J,k)$. Compact homogeneous K\"ahler manifolds have been classified and, in the simply connected case, the following well known result holds (cf. \cite[Theorem 8.89]{Besse} and also \cite{Alek,BFR}):
		
\begin{theorem}\label{Thm:gen.flag}
	Compact, simply connected, homogeneous K\"ahler manifolds are isomorphic, as homogeneous complex manifolds, to the so-called \textit{generalized flag manifolds}, namely to an orbit of the adjoint representation of the connected group of isometries endowed with the canonical complex structure.
\end{theorem}
		
In particular, it can be showed that the connected group of isometries acting on $M$ has to be semisimple. In fact, following \cite{Arv} and \cite{BFR}, generalized flag manifolds can be described as the quotient space $G/K$, where $G$ is a connected compact semisimple Lie group, and $K=C(S)$ is the centralizer of a torus $S$ in $G$. Since $G$ is compact and semisimple, the Killing form $\Kill$ is negative definite on the Lie algebra $\g$ of $G$, thus giving rise to a reductive decomposition $\g=\Liek\oplus\mathfrak{m}$, so that $[\Liek,\Liek]\subset\Liek$ and $[\Liek,\m]\subset\m$, being $\mathfrak{m}=\Liek^\perp$ with respect to $\Kill$. On the other hand, considered the maximal torus $T$ in $G$, containing $S$, the complexification $\h=\mathfrak{t}^\C$ of its Lie algebra $\mathfrak{t}$ is a Cartan subalgebra of $\g^\C$ and, denoting by $\Delta$ be the corresponding root system, $\g^\C$ decomposes as
	$$\g^\C=\h\oplus\bigoplus_{\alpha\in\Delta}\g_\alpha=\h\oplus\bigoplus_{\alpha\in\Delta}\C E_\alpha,$$
where the generators $E_\alpha$ can be chosen in such a way that
	\begin{equation}\label{eq:generators}
		\Kill(E_\alpha,E_\beta)=\begin{cases}
			1& \text{if}\ \beta=-\alpha\\
			0& \text{if}\ \alpha+\beta\neq0
		\end{cases}.
	\end{equation}
Moreover, since $\h\subset\Liek^\C$, there exists a subset $\Delta_K\subset \Delta$ such that:
	$$\Liek^\C=\h\oplus\bigoplus_{\alpha\in\Delta_K}\C E_\alpha,\quad
		\mathfrak{m}^\C=\bigoplus_{\alpha\in\Delta_M}\C E_\alpha$$
where $\Delta_M:=\Delta\setminus\Delta_K$.
		
Now, $K$ is connected as it is the centralizer of a torus. Then, invariant almost complex structures on $G/K$ are in one-to-one correspondence with $\ad(\Liek)$-invariant almost complex structures on $\mathfrak{m}$, that is $J:\m\to\m$ such that $J^2=-I$ and $J[X,U]_\m=[JX,U]_\m$ for every $X\in\m$, $U\in\Liek$. Moreover an almost complex structure on $G/K$ is integrable if and only if the corresponding $J:\m\to\m$ satisfies
	$$[JX,JY]_\m-[X,Y]_\m-J[X,JY]_\m-J[JX,Y]_\m=0$$ 
for every $X,Y\in\m$ (see \cite[Chap. X, Proposition 6.5]{KN}). Then, one can show that their $\C$-linear extensions on $\m^\C$ are in one-to-one correspondence with invariant orderings in $\Delta_M$, and they can be canonically expressed by
	$$JE_\alpha=i\epsilon_\alpha E_\alpha,\quad \alpha\in\Delta_M,$$
where $\epsilon_\alpha=\operatorname{sgn}(\alpha)$ and $\epsilon_{-\alpha}=-\epsilon_\alpha$. In particular, taking into account \eqref{eq:generators}, one has that the Killing form $\Kill$ is $J$-invariant, i.e. $\Kill(JX,JY)=\Kill(X,Y)$ for every $X,Y\in\m$.
		
Similarly, $G$-invariant $2$-forms on $M^{2n}=G/K$ are in one-to-one correspondence with $\ad(\Liek)$-invariant $2$-forms on $\m$. When they are closed, we have the following:
		
\begin{proposition}\label{Prop:2-forms}
	For every closed, $\ad(\Liek)$-invariant $2$-form $\omega$ on $\m$ there exists $Z_\omega\in\mathfrak{z}(\Liek)$ such that
		$$\omega(X,Y)=\Kill([X,Y],Z_\omega)=\Kill([Z_\omega,X],Y)$$
	for every $X,Y\in\m$. In particular, $\omega$ is of type $(1,1)$.
\end{proposition}
\begin{proof}
	The first part of the statement is already known (see for instace \cite{Alek,BFR}),  but we give the proof for the sake of completeness. Let $\omega$ be $\ad(\Liek)$-invariant $2$-form on $\m$, that is $\omega(\ad_UX,Y)+\omega(X,\ad_UY)=0$ for every $U\in\Liek$ and $X,Y\in\m$. Being the Killing form nondegerate, there exists a skew-symmetric endomorphism $\phi:\m\to\m$ such that $\omega(X,Y)=\Kill(\phi X,Y)$ for every $X,Y\in\m$. Then, $\omega$ is closed if and only if 
		$$\omega([X,Y]_\m,Z)+\omega([Y,Z]_\m,X)+\omega([Z,X]_\m,Y)=0,$$
	or equivalently
		$$\Kill([X,Y],\phi Z)+\Kill([Y,Z],\phi X)+\Kill([Z,X],\phi Y)=0$$
	for every $X,Y,Z\in\m$. Let us extend $\omega$ to the whole Lie algebra $\g$ by setting $\omega(U,\cdot)=0$ for every $U\in\Liek$, so that $\phi|_\Liek=0$. Then, using the $\ad(\Liek)$-invariance of $\omega$, one can easily verify that the last condition is still true for every $X,Y,Z\in\g$. Therefore, 
	\begin{align*}		
		\Kill(\phi[X,Y],Z)&=\Kill([Y,Z],\phi X)+\Kill([Z,X],\phi Y)\\
			&=-\Kill(Z,[Y,\phi X])-\Kill([X,Z],\phi Y)\\
			&=\Kill([\phi X,Y],Z)+\Kill([X,\phi Y],Z)
	\end{align*}
	for every $X,Y,Z\in\g$. It follows that $\phi[X,Y]=[\phi X,Y]+[X,\phi Y]$, namely $\phi$ is a derivation of $\g$. Since $\g$ is a semisimple Lie algebra, $\operatorname{Der}(\g)=\ad_\g$, namely there exists $Z_\omega\in\g$ such that $\phi=\ad_{Z_\omega}$. In particular $\phi|_\Liek=0$ implies that $Z_\omega$ belongs to to center $\mathfrak{z}(\Liek)$ of $\Liek$, and $\omega$ can be written as
	$$\omega(X,Y)=\Kill(\ad_{Z_\omega}X,Y)=\Kill([Z_\omega,X],Y)=\Kill([X,Y],Z_\omega).$$
	Concerning the second part of the statement, since $Z_\omega\in\Liek$ and $[\m,\Liek]\subset\m$, the $\ad(\Liek)$-invariant complex structure $J:\m\to\m$ corresponding to the invariant complex structure on $M$, is such that $[JX,Z_\omega]=J[X,Z_\omega]$, for every $X\in\m$. Applying also the $J$-invariance of the Killing form, for every $X,Y\in\m$ one has:
	\begin{align*}
		\omega(JX,JY)&=\Kill([Z_\omega,JX]JY)=-\Kill([JX,Z_\omega],JY)=-\Kill(J[X,Z_\omega],JY)\\&=-\Kill([X,Z_\omega],Y)=\omega(X,Y).\qedhere
	\end{align*}
\end{proof}
		
This leads to the following:
		
\begin{theorem}\label{Thm:comp.hom.aqS}
	There exist no compact homogeneous aqS manifolds of maximal rank.
\end{theorem}
\begin{proof}
	Let $(M,\f,\xi,\eta,g)$ be a compact homogeneous aqS manifold of maximal rank, so that $\eta$ is a contact form on $M$, invariant by the transitive action of a Lie group $G$. By \cite{BW}, the Reeb vector field $\xi$ is regular, and then, according to Theorem \ref{Thm:Boothby-Wang1}, $M$ is a principal circle bundle over a K\"ahler manifold $(M/\xi,J,k)$ endowed with a symplectic form $\omega$ of type $(2,0)$ such that $d\eta=\pi^*\omega$, being $\pi:M\to M/\xi$ the bundle projection. More precisely, by results of A. D\'iaz Miranda and A. Revent\'os \cite{DM.R.}, the base manifold $M/\xi$ of the Boothby-Wang fibration of a contact homogeneous manifold turns out to be always simply connected.\\
	Now, since $\xi$ is invariant under the $G$-action on $M$, one can induce a transitive action of $G$ on $M/\xi$, with respect to which $\pi$ is equivariant, i.e. $g\pi(x):=\pi(gx)$ for every $g\in G$, $x\in M$. Straightforward verifications show that this $G$-action on $M/\xi$ preserves the K\"ahler structrue $(J,k)$ and the symplectic form $\omega$. Therefore, $(M/\xi,J,k)$ is a simply connected, compact homogeneous K\"ahler manifold, namely it is isomorphic to a generalized flag manifold (by Theorem \ref{Thm:gen.flag}), and $\omega$ is a non zero, closed invariant $2$-form of type $(2,0)$.  But this is in contrast with Proposition \ref{Prop:2-forms}.
\end{proof}		
		
\begin{remark}
	In particular, from the above theorem it follows that there exist no compact homogeneous double aqS-Sasakian manifolds (since they have maximal rank). Actually, this is a direct consequence of the fact that compact homogeneous double aqS-Sasakian manifolds fiber on compact homogeneous hyperk\"ahler manifolds, which are isomorphic to flat tori \cite[Theorem 1.b]{Alek1}. This is in contrast with the above mentioned result in \cite{DM.R.}, according to which every compact homogeneous contact manifold fibers onto a simply connected symplectic manifold.
\end{remark}	
\medskip

Now, we provide a topological obstruction to the existence of aqS structures on compact manifolds, even in the non homogeneous case, and for any rank.
		
\begin{proposition}\label{Prop:Betti}
	If $(M^{2n+1},\f,\xi,\eta,g)$ is a compact anti-quasi-Sasakian manifold, then the Betti numbers $b_2, b_{2n-1}$ are strictly positive.
\end{proposition}
\begin{proof}
	First we observe that $d\eta\wedge\Phi^{n-1}=0$. Indeed, it is enough to evaluate the $2n$-form on a local orthornormal $\f$-basis $\{\xi,e_i,\f e_i\}$ ($i=1,\dots,n$). Then the claim is consequence of the facts that $\Phi(\xi,\cdot)=d\eta(\xi,\cdot)=0$ and, for every $i,j=1,\dots,n$,
	$$\Phi(e_i,e_j)=0,\quad \Phi(e_i,\f e_j)=-\delta_{ij},\quad d\eta(e_i,\f e_i)=0,$$
	where the last identity is due to the $\f$-anti-invariance of $d\eta$.
	Now, being $\Phi$ closed, it follows that 
		\begin{equation}\label{eq:eta^Phi}
			d(\eta\wedge\Phi^{n-1})=d\eta\wedge \Phi^{n-1}- \eta\wedge d(\Phi^{n-1}) =0.
		\end{equation} 
	Assuming $b_2=0$, i.e. $H^2(M,\R)$ trivial, we have $\Phi=d\alpha$ for some $\alpha\in \Lambda^1(M)$. 
	Since $\eta\wedge\Phi^{n}$ is a volume form, using \eqref{eq:eta^Phi} and applying the Stokes' theorem, we get:
		$$0\neq\int_M\eta\wedge\Phi^{n}=\int_M(\eta\wedge\Phi^{n-1})\wedge\Phi=\int_M (\eta\wedge\Phi^{n-1})\wedge d\alpha=-\int_M d(\eta\wedge\Phi^{n-1}\wedge \alpha)=0,$$
	which gives a contradiction. Thus, it must be $b_2>0$ and, by the Poincaré duality, one also gets $b_{2n-1}>0$. 
\end{proof}
		
As an immediate consequence of the above proposition one gets the following:
		
\begin{corollary}\label{Thm:sfere}
	Every odd dimensional sphere does not admit any aqS structure.
\end{corollary}
		
\begin{remark}
	The fact that for a compact aqS manifold the second Betti number is greater than $0$ marks a difference with respect to the Sasakian case, where the vanishing of  $b_2$ can be consequence of some curvature conditions. For instance, it is known that every compact Sasakian manifold of strictly positive curvature has vanishing second Betti number (see \cite[Section 6.8]{Blair} and references therein). It follows that a double aqS-Sasakian manifold cannot have positive sectional curvature. This is coherent with the fact that the scalar curvature of such manifolds is $s=-4n<0$ (see \cite[Theorem 4.12]{DiP.D.}).
\end{remark}

\section{Anti-quasi-Sasakian structures on nilpotent Lie groups}
\label{Sec:nilpotent}

In this section we will prove a classification result for nilpotent Lie groups endowed with invariant anti-quasi-Sasakian structures of maximal rank. First, we recall the construction of aqS structures on a class of $2$-step nilpotent Lie groups (see \cite[Section 3.2 ]{DiP.D.}).

\begin{example}[Weighted Heisenberg Lie groups]\label{Ex:w-Heisenberg}
	For any $\lambda=(\lambda_1,\dots,\lambda_n)\in\R^n$, consider the $(4n+1)$-dimensional Lie algebra   $\h^{4n+1}_\lambda=\operatorname{span}\{\xi,\tau_r,\tau_{n+r},\tau_{2n+r},\tau_{3n+r}\ |\ r=1,\dots,n\}$, such that the only eventually nonvanishing commutators are
	\begin{equation}\label{eq:brackets}
		[\tau_r,\tau_{3n+r}]=[\tau_{n+r},\tau_{2n+r}]=2\lambda_r\xi.
	\end{equation}
	Let us denote by $H^{4n+1}_\lambda$ the simply connected Lie group with Lie algebra $\h^{4n+1}_\lambda$. 
	We consider three left invariant almost contact metric structures $(\f_i,\xi,\eta,g)$, $i=1,2,3$, where $g$ is the Riemannian metric with respect to which the basis is orthonormal, $\eta$ is the dual form of $\xi$, and $\f_i$ is given by
	$$\f_i=\sum_{r=1}^{n}\big(\theta_r\otimes\tau_{in+r}-\theta_{in+r}\otimes\tau_r+\theta_{jn+r}\otimes\tau_{kn+r}-\theta_{kn+r}\otimes\tau_{jn+r}\big),$$
	where $\theta_l$ ($l=1,\dots,4n$) is the dual 1-form of $\tau_l$, and $(i,j,k)$ is an even permutation of $(1,2,3)$. 
	The fundamental 2-forms are
	$$\Phi_i=-\sum_{r=1}^{n}(\theta_r\wedge\theta_{in+r}+\theta_{jn+r}\wedge\theta_{kn+r}).$$
	Since the 1-forms $\theta_l$ are all closed by \eqref{eq:brackets}, then $d\Phi_i=0$ for every $i=1,2,3$. Furthermore,
	$$d\eta=-2\sum_{r=1}^n\lambda_r(\theta_r\wedge\theta_{3n+r}+\theta_{n+r}\wedge\theta_{2n+r}),$$
	so one can easily verify that the structures $(\f_1,\xi,\eta,g)$ and $(\f_2,\xi,\eta,g)$ are anti-quasi-Sasakian, while $(\f_3,\xi,\eta,g) $ is quasi-Sasakian.  For each structure, the rank is equal to $4p+1$, if $p$ is the number of the non zero coefficients.
	
	We call $H^{4n+1}_\lambda$ the \textit{weighted Heisenberg Lie group}, with weights $\lambda=(\lambda_1,\dots,\lambda_n)$.	In particular, when the weights are all equal to 1, then $\h^{4n+1}_\lambda=\h^{4n+1}$ is the real Heisenberg Lie algebra of dimension $4n+1$, and $(\f_3,\xi,\eta,g)$ is the standard Sasakian structure on it, being $d\eta=2\Phi_3$. Actually, in this case $(\f_i,\xi,\eta,g)$ ($i=1,2,3$) is a double aqS-Sasakian structure in the sense of Definition \ref{Def:double aqS-S}.
	
	Being $H^{4n+1}_\lambda$ 2-step nilpotent, by a result of A. I. Malcev \cite{Malcev}, when all the weights $\lambda_r$ are rational numbers, then $H^{4n+1}_\lambda$ admits a cocompact discrete subgroup $\Gamma$, so that an anti-quasi-Sasakian structure on the compact nilmanifold $H^{4n+1}_\lambda/\Gamma$ is induced.	
\end{example}
\medskip

We now establish some general properties of almost contact metric Lie algebras $(\g,\f,\xi,\eta,g)$ of maximal rank.

\begin{proposition}\label{Lemma1}
	Let $(\g,\f,\xi,\eta,g)$ be a Lie algebra endowed with an almost contact metric structure of maximal rank, with $\xi$ is Killing. Then $d\eta(\xi,\cdot)=0$.
	If furthermore the center of $\g$ is non-trivial, then
	\begin{enumerate}
		\item $\z(\g)=\R\xi$;
		\item denoting by $[\cdot,\cdot]_\D$ the component of the Lie bracket of $\g$ along $\D$, $(\D,[\cdot,\cdot]_\D)$ is a Lie algebra, and $\g$ is the $1$-dimensional central extension of $\D$ by the $2$-cocycle $-d\eta$; 
		\item setting $N_\f(X,Y,Z):=g(N_\f(X,Y),Z)$, and assuming 
		\begin{equation}\label{eq:transv.K.}
			d\Phi(X,Y,Z)=N_\f(X,Y,Z)=0\quad \forall X,Y,Z\in\D,
		\end{equation}
		$(\D,[\cdot,\cdot]_\D)$ is a K\"ahler Lie algebra with respect to the structure obtained by restriction.
	\end{enumerate}  
\end{proposition}
\begin{proof}
	The Killing condition on the Lie algebra $\g$ is expressed by $g([\xi,X],Y)+g(X,[\xi,Y])=0$, for every $X,Y\in\g$. Then, for 
	$Y=\xi$ one has $d\eta(\xi,X)=-\eta([\xi,X])=0$. Owing to the nondegeneracy of $d\eta$ on $\D$, one has that $\operatorname{Ker}(d\eta)=\R\xi$. Being $d\eta(X,Y)=-\eta([X,Y])$ for every $X,Y\in\g$, it is clear that $\z(\g)\subset\operatorname{Ker}(d\eta)$; thus, if the center of $\g$ is non-trivial, then $\z(g)=\operatorname{Ker}(d\eta)=\R\xi$.
	Moreover, as vector spaces $\g=\D\oplus\R\xi=\D\oplus\z(\g)$, and $\g/\z(\g)\cong\D$ by means $X+\z(\g)\mapsto X_\D$. Being $\z(\g)$ an ideal of $\g$, $\g/\z(\g)$ is a Lie algebra and the above isomorphism induces the Lie bracket $[\cdot,\cdot]_\D$ on $\D$. Moreover, for every $X,Y\in\D$ one has $$[X,Y]=[X,Y]_\D+\eta([X,Y])\xi=[X,Y]_\D-d\eta(X,Y)\xi.$$
	
	Now, assuming \eqref{eq:transv.K.}, set $J=\f|_\D$ and $k=g|_{\D\times\D}$. Then $(J,k)$ is an almost Hermitian structure on $(\D,[\cdot,\cdot]_\D)$. For every $X,Y\in\D$
	\begin{align*}
		N_J(X,Y)&=[JX,JY]_\D+J^2[X,Y]_\D-J[X,JY]_\D-J[JX,Y]_\D\\
		&=[\f X,\f Y]_\D+\f^2[X,Y]-\f[X,\f Y]-\f [\f X,Y]\\
		&=N_\f(X,Y)_\D=0,
	\end{align*}
	and hence $J$ is a complex structure. Finally, let us denote by $\Omega=\Phi|_{\D\times\D}$ the fundamental 2-form associated to the structure $(J,k)$. For every $X,Y,Z\in\D$ we have: 
	\begin{align*}
		d\Omega(X,Y,Z)&=-\Omega([X,Y]_\D,Z)-\Omega([Y,Z]_\D,X)-\Omega([Z,X]_\D,Y)\\
		&=-\Phi([X,Y],Z)-\Phi([Y,Z],X)-\Phi([Z,X],Y)\\
		&=d\Phi(X,Y,Z)=0,
	\end{align*}
	and thus $(\D,J,k)$ is K\"ahlerian.
\end{proof}

\begin{remark}
	In \cite{AFV} it is showed that for a contact Lie algebra $(\g,\eta)$ with non-trivial center, $\z(\g)=\R\xi$, where $\xi$ is the Reeb vector field associated with the contact form $\eta$, uniquely determined by $\eta(\xi)=1$ and $d\eta(\xi,\cdot)=0$. In the first part of Proposition \ref{Lemma1}, we are claiming that, for an almost contact metric structure of maximal rank $(\f,\xi,\eta,g)$, whenever $\xi$ is Killing, it coincides with the Reeb vector field associated with the contact form $\eta$.
\end{remark}
\medskip

\noindent Proposition \ref{Lemma1} applies, in particular, to
\begin{itemize}
	\item anti-quasi-Sasakian Lie algebras ($d\Phi=0$, $N_\f=2d\eta\otimes\xi$), in which case $d\eta$ is $\f$-anti-invariant;
	\item quasi-Sasakian Lie algebras ($d\Phi=0$, $N_\f=0$), in which case $d\eta$ is $\f$-invariant.
\end{itemize} 
Conversely, given a K\"ahler Lie algebra $(\h,[\cdot,\cdot]_\h,J,k)$, if $\omega$ is a $2$-cocycle on $\h$,  one can consider the $1$-dimensional central extension of $\h$, $\g:=\h\oplus\R\xi$, whose Lie bracket $[\cdot,\cdot]$ is defined by
	$$[X,\xi]=0,\quad [X,Y]=[X,Y]_\h-\omega(X,Y)\xi$$
for every $X,Y\in\h$. Then $\g$ is naturally endowed with an almost contact metric structure $(\f,\xi,\eta,g)$, where $\eta$ is the dual $1$-form of $\xi$, and $(\f,g)$ are obtained extending the K\"ahler structure $(J,k)$ in such a way that $\h=\operatorname{Ker}\eta=\operatorname{Im}\f$, $\xi$ is unitary and orthogonal to $\h$. One can easily see that $(\f,\xi,\eta,g)$ satisfies \eqref{eq:transv.K.}. Indeed, the fundamental $2$-form $\Phi=g(\cdot,\f\cdot)$ is closed because so is the fundamental $2$-form $\Omega$ on $\h$, and $\xi$ is a central element. Moreover, for every $X,Y\in\h$, $d\eta(\xi,X)=-\eta([\xi,X])=0$ and $d\eta(X,Y)=-\eta([X,Y])=\omega(X,Y)$. Therefore, 
	$$N_\f(\xi,X)=\f^2[\xi,X]-\f[\xi,\f X]+d\eta(\xi,X)\xi=0$$ 
and
\begin{align*}
	N_\f(X,Y)&=[\f X,\f Y]+\f^2[X,Y]-\f[X,\f Y]-\f[\f X,Y]+ d\eta(X,Y)\xi\\
	&=[JX,JY]_\h-\omega(JX,JY)\xi-[X,Y]_\h-J[X,JY]_\h-J[JX,Y]_\h+\omega(X,Y)\xi\\
	&=-\omega(JX,JY)\xi+\omega(X,Y)\xi.
\end{align*}
In particular, 
\begin{itemize}
		\item if $\omega(JX,JY)=-\omega(X,Y)$, then $N_\f=2d\eta\otimes\xi$ and the $(\f,\xi,\eta,g)$-structure is anti-quasi-Sasakian;
		\item if $\omega(JX,JY)=\omega(X,Y)$, then $N_\f=0$ and the $(\f,\xi,\eta,g)$-structure is quasi-Sasakian; it is Sasakian if and only if $\omega=\Omega$.  
	\end{itemize}

\begin{theorem}\label{Thm:aqS-nil.Lie.alg.}
	Let $(\g,\f,\xi,\eta,g)$ a nilpotent anti-quasi-Sasakian Lie algebra of maximal rank. Then it is isomorphic to a weigthed Heisenberg Lie algebra. In particular, it carries another aqS structure $(\f_2,\xi,\eta,g)$ and a quasi-Sasakian structure $(\f_3,\xi,\eta,g)$ such that $\f\f_2=\f_3=-\f_2\f$.
\end{theorem}
\begin{proof}
	Being a nilpotent Lie algebra, $\g$ has non trivial center. Moreover, since $\xi$ is Killing, by Proposition \ref{Lemma1} we have that $\z(\g)=\R\xi$ and $\D\cong \g/\z(\g)$ is a nilpotent K\"ahler Lie algebra. By a result of Hano \cite{Hano} $\D$ is flat, and \cite[Theorem 2.4]{Milnor} ensures that it is abelian. Therefore, for every $X,Y\in\D$
	$$[X,Y]=-d\eta(X,Y)\xi.$$ Considering an orthonormal basis  $\{\xi,e_i,e_{n+i},e_{2n+i},e_{3n+i}\}$ as in Remark \ref{Rmk:basis}, being $d\eta=2\Psi$, from equation \eqref{eq:Psi} we deduce that the only non-vanishing Lie brackets in $\g$ are
	$$[e_i,e_{3n+i}]=[e_{n+i},e_{2n+i}]=2\lambda_i\xi$$
	for every $i=1,\dots,n$. Therefore, setting $\lambda=(\lambda_1,\dots,\lambda_n)$, the linear isomorphism $F$ from $\g$ to the Heisenberg Lie algebra $\h_\lambda^{4n+1}$ that maps the $\xi$ to the Reeb vector field of $\h_\lambda^{4n+1}$ and such that $F(e_i)=\tau_i$ for every $i=1,\dots,n$, 
	is an isomorphism of Lie algebras and aqS structures.
\end{proof}

In terms of Lie groups, the above result gives the following:

\begin{theorem}\label{Thm:nilmanifold}
	Let $(G,\f,\xi,\eta,g)$ be a simply connected, nilpotent Lie group endowed with a left-invariant anti-quasi-Sasakian structure of maximal rank. Then $G$ is isomorphic to a weighted Heisenberg Lie group $H^{4n+1}_\lambda$.
\end{theorem}

As a consequence, a compact nilmanifold $G/\Gamma$ of dimension $4n+1$ admits an anti-quasi-Sasakian structure of maximal rank induced by a left-invariant aqS structure on $G$ if and only $G$ is isomorphic to $H^{4n+1}_\lambda$.

We conclude the section pointing out that an analogous result holds for nilpotent quasi-Sasakian Lie algebras of maximal rank. Indeed, on any weighted Heisenberg Lie algebra of dimension $2n+1$, $\h^{2n+1}_\lambda=\operatorname{span}\{\xi,\tau_r,\tau_{n+r}\ |\ r=1,\dots,n\}$, with $\lambda=(\lambda_1,\dots,\lambda_n)$ and nonvanishing Lie brackets $[\tau_r,\tau_{n+r}]=2\lambda_r\xi$, one can define a quasi-Sasakian structure $(\f,\xi,\eta,g)$, by taking as usual $\eta$ as the dual form of $\xi$, $g$ as the scalar product with respect to which $\{\xi,\tau_r,\tau_{n+r}\}$ is an orthonormal basis and $\f=\sum_{r=1}^n(\theta_r\otimes\tau_{n+r}-\theta_{n+r}\otimes\tau_r)$ (see for instance \cite[Section 9.3]{AFF}). 

Now, taking a quasi-Sasakian manifold $(M^{2n+1},\f,\xi,\eta,g)$, we show that there exists a local orthonormal $\f$-basis $\{\xi,e_i,\f e_i\}$ ($i=1,\dots,n$), with respect to which
\begin{equation}\label{eq:deta_qS}
	d\eta=-2\sum_{i=1}^n\lambda_i\e_i\wedge\e_{n+i},
\end{equation}
being $\e_l$ the dual form of $e_l$, for every $l=1,\dots,2n$.  
Indeed, being $\xi$ Killing, we can write $d\eta(X,Y)=2g(X,\psi Y)$, with $\psi=-\nabla\xi$ skew-symmetric endomorphism. Then, by the $\f$-invariance of $d\eta$, one has that $\f\psi=\psi\f$, so that $A:=\f\psi$ is a symmetric operator. Thus, in order to obtain \eqref{eq:deta_qS}, it is sufficient to take $e_i,\f e_i$ as eigenvector fields of $A$ associated with the eigenvalue $-\lambda_i$. 
Therefore, using Proposition \ref{Lemma1} and arguing as in the proof of Theorem \ref{Thm:aqS-nil.Lie.alg.}, we can state the following:
	
\begin{theorem}\label{Thm:qS-nil.}
	Let $(G,\f,\xi,\eta,g)$ be a simply connected, nilpotent Lie group endowed with a left-invariant quasi-Sasakian structure of maximal rank. Then $G$ is isomorphic to a wighted Heisenberg Lie group $H^{2n+1}_\lambda$.
\end{theorem}

\noindent This result generalizes the well known fact that, up to isomorphisms, the only $(2n+1)$-dimensional nilpotent Lie algebra endowed with a Sasakian structure is the Heisenberg Lie algebra $\h^{2n+1}$ (with $\lambda=(1,\dots,1)$) \cite{AFV}.
Notice also that as consequence of Theorem \ref{Thm:qS-nil.}, a compact nilmanifold $G/\Gamma$ of dimension $2n+1$ admits a quasi-Sasakian structure of maximal rank induced by a left-invariant quasi-Sasakian structure on $G$ if and only if $G$ is isomorphic to $H^{2n+1}_\lambda$. This is coherent with \cite{CM.dN.M.Y_Sas.Nil}, where the result is proved in the Sasakian case, even for non invariant structures on the Lie group $G$.

\bigskip
		
{\sc Dario Di Pinto\\[0.5em]
		Dipartimento di Matematica\\
		Università degli Studi di Bari Aldo Moro\\
		Via E. Orabona, 4, 70125, Bari, Italy
	}\\
{\em e-mail: } {\tt dario.dipinto@uniba.it} 


\begin{thebibliography}{}
		\bibitem{AFF}
		I. Agricola, A. C. Ferreira, T. Friedrich, \textit{The classification of naturally reductive homogeneous spaces in dimensions $n\le6$}, Diff. Geom. Appl. \textbf{39} (2015), 59-92.
			
		\bibitem{Alek1}
		D. V. Alekseevsky, \textit{Compact quaternion spaces}, Funkcional. Anal. i Prilo\v{z}en \textbf{2} (1968), no. 2, 11–20.
			
		\bibitem{Alek}
		D. V. Alekseevsky, \textit{Flag manifolds}, 11th Yugoslav Geometrical Seminar (Divčibare, 1996), Zb. Rad. Mat. Inst. Beograd. (N.S.) \textbf{6}(14) (1997), 3–35.
			
		\bibitem{AFV}
		A. Andrada, A. M. Fino, L. Vezzoni, \textit{A class of Sasakian 5-manifolds}, Transform. Groups \textbf{14} (2009), no. 3, 493–512.
			
		\bibitem{Arv}
		A. Arvanitoyeorgos, An introduction to Lie Groups and the Geometry of Homogeneous Spaces, Student Mathematical Library \textbf{22}, Amer. Math. Soc., Providence, RI, (2003).
			
		\bibitem{Besse}
		A. L. Besse, {Einstein manifolds}, Springer-Verlag, Berlin (1987).
			
		\bibitem{Blair}
		D. E. Blair, {Riemannian Geometry of Contact and Symplectic Manifolds}, Second Edition. Progress in Mathematics \textbf{203}, Birkh\"{a}user, Boston, 2010.
			
		\bibitem{Blair_qS}
		D. E. Blair, \textit{The theory of quasi-Sasakian structures}, J. Differential Geom. \textbf{1} (1967), 331-345.
			
		\bibitem{BW}
		W. M. Boothby, H. C. Wang, \textit{On contact manifolds}, Ann. of Math. \textbf{68}(2) (1958), 721–734.
			
		\bibitem{BFR}
		M. Bordemann, M. Forger, H. R\"omer, \textit{Homogeneous Kähler manifolds: paving the way towards new supersymmetric sigma models}, Comm. Math. Phys. \textbf{102} (1986), no. 4, 605–617.
			
		\bibitem{BG}
		C. P. Boyer, K. Galicki, Sasakian Geometry, Oxford University Press, Oxford, 2008.
			
		\bibitem{CM.dN.M.Y_Sas.Nil}
		B. Cappelletti-Montano, A. De Nicola, J. C. Marrero, I. Yudin, \textit{Sasakian nilmanifolds}, Int. Math. Res. Not. (IMRN) (2015), no.15, 6648-6660.
			
		\bibitem{CM.dN.Y_cos}
		B. Cappelletti-Montano, A. De Nicola, I. Yudin, \textit{A survey on cosymplectic geometry}, Rev. Math. Phys. \textbf{25} (2013), no. 10, 55 pp. 
			
		\bibitem{Cortes}
		V. Cortés, \textit{A note on quaternionic K\"ahler manifolds with ends of finite volume}, to appear on Q. J. Math. (2023),  \href{https://doi.org/10.1093/qmath/haad029}{https://doi.org/10.1093/qmath/haad029}.
			
		\bibitem{dAFFU}
		L. C. de Andrés, M. Fernandez, A. Fino, L. Ugarte, \textit{Contact $5$-manifolds with $SU(2)$-structure}, Q. J. Math. \textbf{60} (4) (2009), 429-459.
			
		\bibitem{DM.R.}
		A. Díaz Miranda, A. Reventós, \textit{Homogeneous contact compact manifolds and homogeneous symplectic manifolds}, Bull. Sci. Math. \textbf{106}(2) (1982), no. 4, 337–350.
			
		\bibitem{DiP.D.}
		D. Di Pinto, G. Dileo, \textit{Anti-quasi-Sasakian manifolds}, Ann. Global Anal. Geom. \textbf{64}, 5 (2023), 35 pp.
			
		\bibitem{Hano}
		J. Hano, \textit{On K\"ahlerian homogeneous spaces of unimodular Lie groups}, Amer. J. Math. \textbf{79} (1957), 885-900.
			
		\bibitem{KN}
		S. Kobayashi, K. Nomizu, Foundations of Differential Geometry. Vol. II, Wiley-Interscience, New York · London, 1996.
			
		\bibitem{Malcev}
		A. I. Malcev, \textit{On a class of homogeneous spaces}, Izv. Akad. Nauk SSSR, Ser. Mat. \textbf{13} (1949), 9--32; English translation in Am. Math. Soc. Transl.,  No. 39, (1951).
			
		\bibitem{Milnor}
		J. Milnor, \textit{Curvatures of left invariant metrics on Lie groups}, Adv. Math. \textbf{21} (1976), 293–329.
			
		\bibitem{TV}
		A. Tomassini, L. Vezzoni, \textit{Contact Calabi-Yau manifolds and special Legendrian submanifolds}, Osaka J. Math. \textbf{45} (2008), 127-147.
\end{thebibliography}
\end{document}